\documentclass[11pt]{article}
\usepackage[margin=1.05in]{geometry}
\usepackage{amsmath,amssymb,amsthm,mathtools}
\usepackage{microtype}
\usepackage{enumitem}
\usepackage{booktabs}
\usepackage{array}
\usepackage{hyperref}
\usepackage{xcolor}
\usepackage{bm}
\hypersetup{colorlinks=true,linkcolor=blue!45!black,citecolor=blue!45!black,urlcolor=blue!45!black,
  pdftitle={Gamma-Valued Stein--Tomas Restriction for Schatten Classes in Three Dimensions},
  pdfauthor={Rishad Shahmurov and Veli Shahmurov},
  pdfkeywords={Fourier restriction, Stein--Tomas theorem, Schatten classes, gamma-radonifying operators, Strichartz estimates}}

\newtheorem{theorem}{Theorem}[section]
\newtheorem{proposition}[theorem]{Proposition}
\newtheorem{lemma}[theorem]{Lemma}
\newtheorem{corollary}[theorem]{Corollary}
\newtheorem{remark}[theorem]{Remark}

\newcommand{\R}{\mathbb R}
\newcommand{\C}{\mathbb C}
\newcommand{\E}{\mathbb E}
\newcommand{\ga}{\gamma}
\newcommand{\Tr}{\operatorname{Tr}}

\newcommand{\cF}{\mathcal F}
\newcommand{\cH}{\mathcal H}
\newcommand{\cM}{\mathcal M}
\newcommand{\supp}{\operatorname{supp}}

\title{\bfseries $\gamma$-Valued Stein--Tomas Restriction for Schatten Classes}
\author{Rishad Shahmurov\thanks{Cellular Products Research and Development, 595 East Crossville Rd., Roswell, GA 30075, USA. Email: \texttt{rshahmurov@crimson.ua.edu}.}
\and Veli Shahmurov\thanks{Antalya Bilim University, Department of Industrial Engineering, 07190 Antalya, Turkey; Azerbaijan State Economic University, AZ1001 Baku, Azerbaijan. Emails: \texttt{veli.sahmurov@antalya.edu.tr}, \texttt{veli.sahmurov@gmail.com}.}}
\date{}

\begin{document}
\maketitle

\begin{abstract}
Let $\Sigma\subset\mathbb R^3$ be a compact $C^2$ hypersurface with nonvanishing Gaussian curvature. We prove the sharp matrix-dimension-independent Schatten-valued Stein--Tomas range
\[
R_\Sigma:L^p(\mathbb R^3;S^a)\longrightarrow\gamma(L^2(\Sigma),S^a),
\qquad 1<a<\infty,\quad 1\le p\le\min\{a,4/3\}.
\]
The core endpoint is $E_\Sigma:\gamma(L^2(\Sigma),S^4)\to L^4(\mathbb R^3;S^4)$. After squaring the quartic norm, multiplication lands in the Hilbert--Schmidt class. A local coarea argument splits resonant quartets according to the shorter adjacent frequency chord, producing a dimension-free row/column estimate for elliptic and hyperbolic curvature. We also show that Gaussian type $4/3$ alone does not characterize the corresponding endpoint for arbitrary Banach spaces. Applications include matrix Herglotz waves, a two-dimensional Schr\"odinger $L^4$ estimate, and free spectral-measure bounds.
\end{abstract}



\section{Introduction}

\subsection{The restriction problem and the Stein--Tomas line}
The Fourier restriction problem asks whether the Fourier transform of an $L^p$ function, which a priori is defined only almost everywhere in the ambient frequency space, can be meaningfully restricted to a lower-dimensional curved set.  For a smooth compact hypersurface $\Sigma\subset\R^n$ with surface measure $d\sigma$, the classical $L^2$ restriction problem is
\begin{equation}\label{eq:scalar-restriction}
 \|\widehat f\big|_\Sigma\|_{L^2(\Sigma)}
 \le C\|f\|_{L^p(\R^n)}.
\end{equation}
The curvature-sensitive theory emerged through Stein's early restriction observation, the Fefferman--Stein bounds recorded in Fefferman's work on strongly singular convolution operators \cite{Fefferman1970}, and Tomas's 1975 theorem \cite{Tomas1975}.  Tomas obtained the full nonendpoint range and explicitly noted that Stein had supplied the endpoint by complex interpolation; see the historical account in \cite{Oliveira2024} and Stein's monograph \cite{Stein1993}.  For the sphere, the Stein--Tomas range is
\[
 1\le p\le p_{\mathrm{ST}}:=\frac{2(n+1)}{n+3}.
\]
At $n=3$ this becomes $p_{\mathrm{ST}}=4/3$.  The Knapp concentration example shows that the exponent cannot be improved in the scalar $L^p\to L^2$ scale.

The dual extension operator is
\begin{equation}\label{eq:extension}
 E_\Sigma g(x)
 =\int_\Sigma e^{ix\cdot\omega}g(\omega)\,d\sigma(\omega),
\end{equation}
and the scalar endpoint in three dimensions is
\begin{equation}\label{eq:scalar-ST4}
 \|E_\Sigma g\|_{L^4(\R^3)}\lesssim_\Sigma \|g\|_{L^2(\Sigma)}.
\end{equation}
For the sphere $S^2$, the sharp endpoint and its extremizers have been studied in detail; in particular, Foschi proved that constants are global maximizers for the adjoint sphere inequality \cite{Foschi2015}, and the even exponent $4$ makes the convolution structure especially transparent \cite{Oliveira2024}.

The restriction viewpoint quickly entered PDE.  Strichartz extended the extension method to quadratic surfaces and connected it explicitly with spacetime decay for wave and Klein--Gordon type equations \cite{Strichartz1977}.  Carleson--Sj\"olin oscillatory integrals \cite{CarlesonSjolin1972} and Stein's oscillatory-integral theory enlarged the framework from translation-invariant surfaces to variable phases.  Sogge's spectral-cluster estimates on compact manifolds \cite{Sogge1988} and the uniform Sobolev inequalities of Kenig--Ruiz--Sogge \cite{KRS1987} show another major route from restriction geometry to spectral and elliptic PDE.  Keel--Tao's abstract endpoint Strichartz theorem \cite{KeelTao1998} systematized the dispersive consequences.

A second revolution came from bilinear and multilinear restriction.  Tao--Vargas--Vega \cite{TaoVargasVega1998}, Tao's sharp bilinear paraboloid estimate \cite{Tao2003}, the multilinear restriction theorem of Bennett--Carbery--Tao \cite{BCT2006}, the Bourgain--Guth multilinear approach \cite{BourgainGuth2011}, and Guth's polynomial-partitioning work \cite{Guth2018} established restriction as a central meeting point of harmonic analysis, incidence geometry, and dispersive PDE.  These developments concern the scalar restriction conjecture beyond the Stein--Tomas line.  Our problem is different: the ambient exponent remains exactly the classical Stein--Tomas exponent, while the coefficient space becomes noncommutative.

\subsection{Why a coefficient-space theorem is needed: a multichannel dispersive model}
A simple PDE model already shows why a componentwise scalar estimate does not capture the problem studied here.  Consider a matrix-valued dispersive band
\begin{equation}\label{eq:motivating-flow}
 U(t,x)=\int e^{i(x\cdot\xi-t\lambda(\xi))}G(\xi)\,d\xi,
 \qquad G(\xi)\in S^4,
\end{equation}
where the phase $\lambda$ is one branch of a constant-coefficient system and has nondegenerate Hessian on the frequency patch under consideration.  In a system with frequency-dependent polarization one naturally encounters
\[
 G(\xi)=P(\xi)A(\xi)Q(\xi),
\]
with spectral projectors or channel maps $P(\xi),Q(\xi)$ that cannot be simultaneously diagonalized as $\xi$ varies.  For example, even in two internal dimensions the rank-one projections
\[
 P(\xi)=v(\xi)v(\xi)^*,
 \qquad
 v(\xi)=\big(\cos\vartheta(\xi),\sin\vartheta(\xi)\big)^T,
\]
generically satisfy $[P(\xi),P(\eta)]\ne0$.  Estimating every matrix entry separately and then summing either loses the coherent channel geometry or introduces constants depending on the number of channels.  The natural invariant quantity is instead the Gaussian row/column norm of the whole amplitude.  On a nondegenerate graph, the principal estimate proved below gives precisely
\[
 \|U\|_{L^4_{t,x}(S^4)}
 \lesssim
 \|G\|_{\gamma(L^2_\xi,S^4)},
\]
with no dependence on matrix dimension.  Thus the theorem is designed for dispersive systems in which the propagating field itself is operator valued, rather than for a scalar field copied into several independent channels.

A useful second geometry, very different from the sphere, is the saddle
\begin{equation}\label{eq:hyperbolic-model}
 \phi_{\rm hyp}(\xi)=\xi_1^2-\xi_2^2.
\end{equation}
Its Hessian has determinant $-4$, so compact patches have strictly negative Gaussian curvature and fall directly under the local theorem.  The scalar restriction theory for hyperbolic surfaces is classical and has been developed much further in the modern bilinear and polynomial-partitioning literature; see, for example, Buschenhenke--M\"uller--Vargas \cite{BuschenhenkeMullerVargas2023}.  What is new in the present setting is not the scalar saddle geometry itself, but the dimension-free coherent estimate for noncommuting Schatten-valued amplitudes on that geometry.  The same local theorem also covers genuinely nonquadratic perturbations such as
\[
 \phi_\varepsilon(\xi)=\xi_1^2-\xi_2^2+\varepsilon\xi_1^3
\]
on a sufficiently small patch where the Hessian remains nonsingular.  This example makes clear that the argument does not rely on positive curvature, exact quadratic structure, or spherical symmetry.

\subsection{The new coefficient geometry}
Let $S^a$ be the Schatten class on a separable Hilbert space.  We study
\begin{equation}\label{eq:gamma-restriction-intro}
 R_\Sigma f:=\widehat f\big|_\Sigma:
 L^p(\R^3;S^a)
 \longrightarrow
 \ga(L^2(\Sigma),S^a).
\end{equation}
The target is not the Bochner space $L^2(\Sigma;S^a)$.  It is the space of $\gamma$-radonifying operators from the Hilbert space $L^2(\Sigma)$ into $S^a$, the natural Gaussian replacement for Hilbert-valued square functions; see \cite{vanNeerven2010,HNVW2017}.  For a finite orthonormal system $(h_j)$ and coefficients $(x_j)\subset S^a$,
\[
 \left\|\sum_jh_j\otimes x_j\right\|_\ga
 =\left(\E\left\|\sum_j\gamma_jx_j\right\|_{S^a}^2\right)^{1/2}.
\]
For $a\ge2$, noncommutative Khintchine identifies this Gaussian norm with simultaneous row and column square functions \cite{LustPiquard1986,LustPiquardPisier1991}:
\begin{equation}\label{eq:nck-intro}
 \left\|\sum_jh_j\otimes x_j\right\|_\ga
 \asymp_a
 \max\left\{
 \left\|(\sum_jx_j^*x_j)^{1/2}\right\|_a,
 \left\|(\sum_jx_jx_j^*)^{1/2}\right\|_a
 \right\}.
\end{equation}
Thus the first critical non-Hilbert endpoint in dimension three is dual to
\begin{equation}\label{eq:main-extension-intro}
 E_\Sigma:\ga(L^2(\Sigma),S^4)
 \longrightarrow L^4(\R^3;S^4).
\end{equation}

The key point is that the power $4$ is perfectly matched to the Schatten product law
\[
 S^4\cdot S^4\subset S^2,
\]
and $S^2$ is Hilbert.  The scalar identity ``square the $L^4$ norm and use Plancherel'' therefore survives, but the order of matrix multiplication matters.  A fixed ordering produces a genuine crossing commutator.  The proof below avoids estimating that commutator separately.  Instead, on each resonant quartet we compare two adjacent frequency chords and pair the Hilbert--Schmidt factors along the longer one.  On the region where one chord is shorter, its length controls the resonance fiber while the longer chord controls the coarea denominator; the resulting fiber mass is uniformly bounded.  This direct two-region decomposition is the geometric mechanism behind the dimension-free estimate.

\subsection{Main results}
Our first theorem is the dimension-free matrix endpoint.  We refer to it as the Schatten-valued Stein--Tomas extension theorem; the notation $\ga(L^2(\Sigma),S^4)$ is independent of the particular finite-dimensional realization of $S^4$.

\begin{theorem}\label{thm:main-extension}
Let $\Sigma\subset\R^3$ be a compact $C^2$ hypersurface with nonvanishing Gaussian curvature and smooth positive density $d\sigma$.  Then
\begin{equation}\label{eq:main-extension}
 \boxed{
 E_\Sigma:\ga(L^2(\Sigma),S^4)
 \longrightarrow L^4(\R^3;S^4)
 }
\end{equation}
is bounded.  The constant depends on finitely many local $C^2$ geometric bounds, a positive lower curvature bound, the density, and a finite atlas, but not on matrix dimension.
\end{theorem}

By $\gamma$-trace duality this gives the critical restriction corollary, which is the small-Schatten endpoint used in the classification.

\begin{corollary}\label{cor:critical-restriction}
Under the assumptions of Theorem~\ref{thm:main-extension},
\begin{equation}\label{eq:critical-restriction}
 \boxed{
 R_\Sigma:
 L^{4/3}(\R^3;S^{4/3})
 \longrightarrow
 \ga(L^2(\Sigma),S^{4/3})
 }
\end{equation}
is bounded with a matrix-dimension-independent constant.
\end{corollary}

Combining the critical endpoint with the Hilbert, large-Schatten, and universal $L^1$ branches gives the sharp Schatten Stein--Tomas range in $\R^3$.

\begin{theorem}\label{thm:classification}
Let $1<a<\infty$.  Then
\begin{equation}\label{eq:classification}
 R_\Sigma:L^p(\R^3;S^a)
 \longrightarrow\ga(L^2(\Sigma),S^a)
\end{equation}
is bounded for
\begin{equation}\label{eq:range}
 \boxed{1\le p\le \min\{a,4/3\}.}
\end{equation}
For every such curved surface the scalar rank-one Knapp example forces $p\le4/3$.  For $a<2$, a dimension-independent estimate on the infinite Schatten class (equivalently, estimates uniform over all matrix dimensions) also forces $p\le a$ by Gaussian type.  Thus \eqref{eq:range} is the sharp matrix-dimension-independent range.
\end{theorem}

For the paraboloid the proof is especially simple.  The following corollary is the two-dimensional Schr\"odinger $L^4$ estimate with $\gamma$-radonifying Schatten data.

\begin{corollary}\label{cor:schrodinger}
For finite-rank Schwartz $S^4$-valued data $f$ on $\R^2$,
\begin{equation}\label{eq:schrodinger}
 \boxed{
 \|e^{it\Delta}f\|_{L^4(\R_t\times\R_x^2;S^4)}
 \lesssim
 \|\widehat f\|_{\ga(L^2(\R^2),S^4)}.
 }
\end{equation}
The estimate extends by completion to the corresponding $\gamma$-data space.
\end{corollary}

\subsection{At a glance: what the generalization changes}
The following table separates exact scalar recovery from genuinely new coefficient-space information.

\begin{center}
\small
\renewcommand{\arraystretch}{1.18}
\begin{tabular}{>{\raggedright\arraybackslash}p{0.20\textwidth}>{\raggedright\arraybackslash}p{0.32\textwidth}>{\raggedright\arraybackslash}p{0.38\textwidth}}
\toprule
Classical object & Classical statement & Present coefficient-space statement \\
\midrule
Stein--Tomas in $\R^3$ & $L^{4/3}(\R^3)\to L^2(\Sigma)$ & $L^{4/3}(\R^3;S^{4/3})\to\ga(L^2(\Sigma),S^{4/3})$ \\
Herglotz/Helmholtz & $L^2(S^2)\to L^4_x$ & coherent noncommuting angular amplitudes $\ga(L^2,S^4)\to L^4_x(S^4)$ \\
$2D$ Schr\"odinger & $\|e^{it\Delta}f\|_{L^4_{t,x}}\lesssim\|f\|_2$ & $\|e^{it\Delta}f\|_{L^4_{t,x}(S^4)}\lesssim\|\widehat f\|_\ga$ \\
Free spectral measure & scalar $L^{4/3}\to L^4$ density estimate & $S^{4/3}$-valued spectral density with the same scaling \\
Gaussian surface fields & scalar/Hilbert random amplitudes & radonification of coherent $S^4$-valued cylindrical data \\
Variable phase / resolvent & Carleson--Sj\"olin and KRS theories & not claimed here; these are the next coefficient-space problems \\
\bottomrule
\end{tabular}
\end{center}

The first three rows are direct specializations or consequences of the theorem.  The last row is deliberately separated: a compact-surface theorem does not by itself provide variable-phase summation or the principal-value cancellation of a resolvent.

\subsection{Relation to trace-ideal and noncommutative restriction results}
Frank and Sabin proved orthonormal-function analogues of Stein--Tomas and Strichartz, together with Schatten trace-ideal versions of uniform Sobolev estimates and applications to limiting absorption, Hartree dynamics, Lieb--Thirring bounds, and scattering \cite{FrankSabin2017}.  Their restriction theorem controls scalar densities generated by orthonormal systems.  The present theorem controls an operator-valued surface amplitude itself through its two-sided row/column Gaussian geometry.  Neither formulation formally contains the other.

Likewise, Fourier restriction on quantum Euclidean spaces \cite{HongLaiWang2023} and Fourier--Wigner restriction for Schatten operators \cite{LuefSamuelsen2025} are genuinely noncommutative restriction theories, but the Fourier transforms there are noncommutative/phase-space transforms.  M\"uller recently gave a pseudodifferential-operator proof of the Fourier--Wigner/classical phase-space equivalence and improved the dependence of the corresponding constants on the support diameter \cite{Muller2026}.  These results remain different from \eqref{eq:gamma-restriction-intro}: here the physical-variable Fourier transform is the ordinary Euclidean transform of an $S^a$-valued function, and the surface target is the Banach-valued $\gamma$ space.

\subsection{How the classical theory is recovered}
Theorem~\ref{thm:main-extension} contains the classical endpoint in two elementary ways.  Setting the coefficient algebra equal to $\C$ gives \eqref{eq:scalar-ST4} directly.  More generally, taking $G(\omega)=g(\omega)x_0$ with a fixed rank-one $x_0$ reduces the estimate to scalar Stein--Tomas.  Interpolation with the trivial $L^1\to L^2(\Sigma)$ restriction estimate recovers the full scalar range $1\le p\le4/3$.

A Hilbert space can also be embedded as a pure column in a Schatten class, so the theorem contains the Hilbert-valued $L^2\to L^4$ extension estimate as a special coefficient configuration.  In the noncommutative setting this pure-row/pure-column case is much easier than the general theorem because only one square function survives.  The genuinely new content is simultaneous left/right interaction.

The PDE consequences should be read similarly.  Corollary~\ref{cor:schrodinger} reduces, for scalar or rank-one data, to the classical two-dimensional $L^4$ Schr\"odinger Strichartz inequality.  The Herglotz-wave consequence in Section~\ref{sec:PDE} reduces to the classical $L^4$ bound for Fourier extension from $S^2$.  By contrast, the cone has one flat principal direction, so the wave-equation restriction problem is not a formal special case of Theorem~\ref{thm:main-extension}; Strichartz's quadratic-surface theory and later cone restriction estimates remain distinct inputs.

\section{Schatten classes and \texorpdfstring{$\gamma$}{gamma}-radonifying operators}\label{sec:prelim}

\subsection{Schatten classes}
Let $\cH$ be a separable Hilbert space.  For $1\le a<\infty$, the Schatten class $S^a(\cH)$ consists of compact operators $A$ whose singular values $s_j(A)$ belong to $\ell^a$:
\[
 \|A\|_{S^a}
 =\left(\sum_js_j(A)^a\right)^{1/a}
 =\big(\Tr|A|^a\big)^{1/a}.
\]
We write $S^\infty=B(\cH)$ when only the operator norm is relevant.  Standard operator-ideal facts used below are collected, for example, in \cite{Simon2005}; the type/cotype statement is classical in the local theory of Banach spaces, see \cite{Tomczak1989}:
\begin{enumerate}[label=(\roman*)]
\item duality $(S^a)^*=S^{a'}$ for $1<a<\infty$ under the trace pairing $\Tr(A^*B)$;
\item Schatten H\"older,
\[
 \|AB\|_{S^r}\le \|A\|_{S^p}\|B\|_{S^q},
 \qquad \frac1r=\frac1p+\frac1q;
\]
\item complex interpolation $[S^{a_0},S^{a_1}]_\theta=S^a$ with $1/a=(1-\theta)/a_0+\theta/a_1$;
\item $S^2$ is the Hilbert--Schmidt space with inner product $\langle A,B\rangle_{HS}=\Tr(A^*B)$;
\item $S^a$ has Gaussian type $\min\{a,2\}$ and Gaussian cotype $\max\{a,2\}$, up to the usual endpoint interpretation;
\item the norm is unitarily invariant, $\|UAV\|_{S^a}=\|A\|_{S^a}$ for unitaries $U,V$, and for an operator of rank at most $N$,
\[
 \|A\|_{S^b}\le \|A\|_{S^a}
 \le N^{1/a-1/b}\|A\|_{S^b},
 \qquad 1\le a\le b\le\infty.
\]
In particular all Schatten norms agree on a rank-one partial isometry.
\end{enumerate}

The scale has familiar operator-theoretic meanings.  $S^1$ is the trace class, the natural home of density matrices and nuclear operators; $S^2$ is the Hilbert--Schmidt class; $S^\infty$ is the bounded-operator endpoint.  The class $S^4$ is the first genuinely non-Hilbert Schatten space for which a quartic norm becomes Hilbertian after one multiplication.  This is exactly the algebraic feature exploited below.

The Hilbert--Schmidt product identity is the reason $a=4$ is special:
\begin{equation}\label{eq:S4HS}
 \|AB^*\|_{S^2}^2
 =\Tr(A^*A B^*B),
 \qquad A,B\in S^4.
\end{equation}
This is exactly the noncommutative analogue of the scalar product appearing after squaring an $L^4$ extension norm.

\subsection{\texorpdfstring{$\gamma$}{gamma}-radonifying operators}
Let $H$ be a Hilbert space and $X$ a Banach space.  If
\[
 T=\sum_{j=1}^N h_j\otimes x_j:H\to X
\]
with $(h_j)$ orthonormal, set
\begin{equation}\label{eq:gamma-def}
 \|T\|_{\ga(H,X)}
 =\left(\E\left\|\sum_{j=1}^N\gamma_jx_j\right\|_X^2\right)^{1/2}.
\end{equation}
The completion of finite-rank operators is $\ga(H,X)$; see \cite{vanNeerven2010,HNVW2017}.  The ideal property states that if $A:H_0\to H$ and $B:X\to Y$ are bounded, then
\begin{equation}\label{eq:gamma-ideal}
 \|BTA\|_{\ga(H_0,Y)}
 \le \|B\|\,\|T\|_\ga\,\|A\|.
\end{equation}
For Hilbert $X$, $\ga(H,X)$ is the Hilbert--Schmidt class.  For non-Hilbert $X$, the norm records Gaussian geometry that cannot in general be replaced by a Bochner $L^2$ norm.

For $2\le a<\infty$, noncommutative Khintchine gives
\begin{equation}\label{eq:NCK}
 \left(\E\left\|\sum_j\gamma_jx_j\right\|_a^2\right)^{1/2}
 \asymp_a
 \max\left\{
 \left\|(\sum_jx_j^*x_j)^{1/2}\right\|_a,
 \left\|(\sum_jx_jx_j^*)^{1/2}\right\|_a
 \right\}.
\end{equation}
For $1<a\le2$, the corresponding norm is the infimum over decompositions $x_j=y_j+z_j$ of the sum of the column norm of $(y_j)$ and the row norm of $(z_j)$; see \cite{LustPiquard1986,LustPiquardPisier1991}.

For an $S^4$-valued simple function $G$ on a measure space, write
\begin{equation}\label{eq:CG-RG}
 C_G=\left(\int G^*G\right)^{1/2},
 \qquad
 R_G=\left(\int GG^*\right)^{1/2},
\end{equation}
and
\begin{equation}\label{eq:Gamma4}
 \|G\|_{\Gamma_4}
 =\max\{\|C_G\|_4,\|R_G\|_4\}.
\end{equation}
Then
\begin{equation}\label{eq:gamma-Gamma4}
 \|G\|_{\ga(L^2,S^4)}\asymp \|G\|_{\Gamma_4}.
\end{equation}

\subsection{Two elementary transfer principles}
We record two facts needed later.  The first is the universal $L^1$ restriction principle, which requires no geometry of the coefficient Banach space.

\begin{proposition}\label{prop:L1}
For every Banach space $X$ and compact finite-measure hypersurface $\Sigma$,
\[
 R_\Sigma:L^1(\R^3;X)\to\ga(L^2(\Sigma),X)
\]
is bounded.
\end{proposition}

\begin{proof}
For $x\in\R^3$, let $k_x(\omega)=e^{-ix\cdot\omega}$.  The restriction operator associated with a simple $X$-valued $f$ is the Bochner integral of rank-one operators
\[
 T_f=\int_{\R^3} k_x\otimes f(x)\,dx.
\]
Since $\|k_x\otimes y\|_\ga=\|k_x\|_{L^2(\Sigma)}\|y\|_X=\sigma(\Sigma)^{1/2}\|y\|_X$, the triangle inequality in $\ga$ yields
\[
 \|T_f\|_\ga\le\sigma(\Sigma)^{1/2}\|f\|_{L^1(X)}.
\]
Density completes the proof.
\end{proof}

The second elementary principle is the Hilbert-output factorization, which transfers a scalar Hilbert-input estimate to Schatten-valued $\gamma$ data on the range $1<a\le2$.

\begin{proposition}\label{prop:factorization}
Let $T:H\to L^q(M)$ be bounded, where $H$ is Hilbert and $2\le q\le\infty$.  For $1<a\le2$,
\[
 T\otimes I_{S^a}:\ga(H,S^a)\to L^q(M;S^a)
\]
is bounded with a constant depending only on $a$ and $\|T\|$.
\end{proposition}

\begin{proof}
It suffices to treat a finite orthonormal family $(h_j)$ and coefficients $x_j$.  Put
\[
 A=(\sum_jx_j^*x_j)^{1/2}.
\]
Choose $r$ by $1/a=1/2+1/r$.  On the support of $A$, write
\[
 B=A^{1-a/2},\qquad y_j=x_jA^{a/2-1},\qquad x_j=y_jB.
\]
Spectral regularization removes the possible kernel.  Then
\[
 \sum_j\|y_j\|_2^2=\Tr A^a,
 \qquad
 \|B\|_r=(\Tr A^a)^{1/r}.
\]
Hilbert-valued tensorization of $T$ and Schatten H\"older imply
\[
 \left\|\sum_j(Th_j)x_j\right\|_{L^q(S^a)}
 \le \|T\|\,\|A\|_a.
\]
A left factorization gives the row estimate.  Taking the infimum over row/column decompositions and using the $a\le2$ form of noncommutative Khintchine proves the claim.
\end{proof}

By duality, Proposition~\ref{prop:factorization} implies the full scalar Stein--Tomas range for restriction with values in $S^a$ when $a\ge2$.

\section{Local quartic resonance geometry}\label{sec:geometry}

The proof of Theorem~\ref{thm:main-extension} is local on the surface.  Let $U\Subset\R^2$ be convex and
\[
 \Gamma(\xi)=(\xi,\phi(\xi)),
\]
where $\phi\in C^2(U)$ and $D^2\phi$ is uniformly nonsingular.  After shrinking $U$ around a fixed point, we may assume that $D^2\phi$ stays uniformly close to one invertible matrix.  We initially work with Lebesgue measure $d\xi$; smooth positive surface densities are inserted in Section~\ref{sec:global}.

For $G:U\to S^4$, extended by zero outside $U$, define
\begin{equation}\label{eq:local-extension}
 E_\phi G(x,t)
 =\int_{\R^2} e^{i(x\cdot\xi+t\phi(\xi))}G(\xi)\,d\xi.
\end{equation}
We first make rigorous the Fourier-side object obtained by squaring this extension.

For $u\in\R^2$ put
\begin{equation}\label{eq:qu}
 q_u(\xi)=\phi(\xi+u)-\phi(\xi).
\end{equation}
Whenever both $\xi$ and $\xi+u$ belong to $U$,
\begin{equation}\label{eq:gradqu}
 \nabla_\xi q_u(\xi)
 =\left(\int_0^1D^2\phi(\xi+tu)\,dt\right)u.
\end{equation}

The next result is the difference-map transversality lemma.  It turns nondegeneracy of the Hessian into a uniform coarea description of each difference fiber.

\begin{lemma}\label{lem:difference-transversality}
After shrinking $U$ if necessary, there is $c_0>0$ such that
\begin{equation}\label{eq:gradqu-lower}
 |\nabla_\xi q_u(\xi)|\ge c_0|u|
\end{equation}
for every admissible $(\xi,u)$.  More precisely, for $u\ne0$ one may choose
\[
 e_u=\frac{H_0u}{|H_0u|},\qquad H_0=D^2\phi(\xi_0),
\]
so that
\begin{equation}\label{eq:directional-qu}
 |e_u\cdot\nabla_\xi q_u(\xi)|\ge c_0|u|
\end{equation}
with one fixed sign throughout the patch.  Consequently every nonempty level set of $q_u$ is, in coordinates adapted to $e_u$, a single uniformly Lipschitz graph; in particular its length inside the fixed compact support is bounded uniformly after $|u|$ is bounded away from zero.
\end{lemma}

\begin{proof}
Fix $\xi_0\in U$ and write $H_0=D^2\phi(\xi_0)$.  If $s_0$ is the smallest singular value of $H_0$, shrink the relevant convex neighborhood until
\[
 \|D^2\phi(\zeta)-H_0\|\le s_0/4.
\]
Then the averaged Hessian in \eqref{eq:gradqu} differs from $H_0$ by at most $s_0/4$, which proves \eqref{eq:gradqu-lower}.  Pairing the same identity with $e_u=H_0u/|H_0u|$ gives \eqref{eq:directional-qu}, after a harmless further shrinking.  Since the directional derivative never changes sign, $q_u$ is strictly monotone along every line parallel to $e_u$; the implicit-function theorem and the upper Hessian bound give the asserted uniform graph and slope bounds.
\end{proof}

For $\varepsilon>0$ define the $S^2$-valued difference density
\begin{equation}\label{eq:Keps}
 K_\varepsilon(u,s)
 =\mathbf 1_{\{|u|\ge\varepsilon\}}
 \int_{\{q_u(\xi)=s\}}
 \frac{G(\xi)^*G(\xi+u)}{|\nabla_\xi q_u(\xi)|}
 \,d\cH^1(\xi).
\end{equation}
For each fixed $\varepsilon>0$, Lemma~\ref{lem:difference-transversality} gives a uniformly Lipschitz description of the $q_u$-fibers for $|u|\ge\varepsilon$.  Since $G$ is bounded and compactly supported, \eqref{eq:Keps} is then bounded on a compact $(u,s)$-set, and hence $K_\varepsilon\in L^2(\R^3;S^2)$.  It is the coarea density of the pushforward of
\[
 \mathbf 1_{\{|\eta-\xi|\ge\varepsilon\}}
 G(\xi)^*G(\eta)\,d\xi\,d\eta
\]
under the difference map
\[
 (\xi,\eta)\longmapsto \Gamma(\eta)-\Gamma(\xi).
\]
Consequently, up to the harmless Fourier-normalization constant,
\begin{equation}\label{eq:Keps-fourier}
 K_\varepsilon
 \longrightarrow
 \cF_{x,t}\big((E_\phi G)^*E_\phi G\big)
\end{equation}
in the sense of tempered distributions as $\varepsilon\downarrow0$.  In fact, for smooth compactly supported finite-rank $G$ the omitted region $|\eta-\xi|<\varepsilon$ has $S^2$-valued total variation tending to zero, so the underlying pushforward measures converge in total variation.

The local theorem will follow once we prove the uniform estimate
\begin{equation}\label{eq:Kuniform}
 \sup_{\varepsilon>0}
 \|K_\varepsilon\|_{L^2(\R^3;S^2)}^2
 \lesssim_\phi
 \left\|\left(\int_U G^*G\right)^{1/2}\right\|_4^4
 +
 \left\|\left(\int_U GG^*\right)^{1/2}\right\|_4^4.
\end{equation}
To identify the geometry in this $L^2$ norm, expand \eqref{eq:Keps} and write the second point on the same $q_u$-fiber as $\xi+v$.  Equality of the two $q_u$ values is precisely
\begin{equation}\label{eq:Phi}
 \Phi_{\xi,v}(u)
 =\phi(\xi+u)-\phi(\xi)-\phi(\xi+u+v)+\phi(\xi+v)=0.
\end{equation}
The four corresponding frequencies are
\begin{equation}\label{eq:uv-param}
 \xi_1=\xi,
 \qquad
 \xi_2=\xi+u,
 \qquad
 \xi_4=\xi+v,
 \qquad
 \xi_3=\xi+u+v.
\end{equation}
To make the re-disintegration precise, fix $\delta>0$ and work first on $|u|,|v|\ge\delta$.  On this regular region the scalar coarea theorem and Fubini define one resonance measure $d\nu_\xi(u,v)$ on $\{\Phi_{\xi,v}(u)=0\}$, and it has the two equivalent iterated representations
\begin{align}
 \int F\,d\nu_\xi
 &=\int_{|u|\ge\delta}\int_{\{\Phi_{\xi,v}(u)=0\}}
   \frac{F(u,v)}{|\nabla_v\Phi_{\xi,v}(u)|}\,d\cH^1(v)\,du \label{eq:res-measure-v}\\
 &=\int_{|v|\ge\delta}\int_{\{\Phi_{\xi,v}(u)=0\}}
   \frac{F(u,v)}{|\nabla_u\Phi_{\xi,v}(u)|}\,d\cH^1(u)\,dv . \label{eq:res-measure-u}
\end{align}
Indeed both sides are the coarea disintegration of the same hypersurface measure for the scalar level set $\Phi=0$; equivalently, this follows by applying coarea in one variable and then Fubini, with all gradients bounded away from zero on the truncated region.  This observation is what permits us below to choose the $u$- or $v$-fiber according to which chord is longer.

If
\[
 A_\xi(u)=G(\xi)^*G(\xi+u),
\]
then the pointwise trace in the expansion of $\|K_\varepsilon\|_2^2$ is
\[
 \Tr\!\big(A_\xi(u)^*A_{\xi+v}(u)\big)
 =\Tr(G_2^*G_1G_4^*G_3).
\]
This is the complex conjugate, after cyclic permutation, of
\[
 \Tr(G_1^*G_2G_3^*G_4).
\]
Hence the two expressions have exactly the same absolute value.  Thus it suffices to control the absolute quartic form on the resonance set \eqref{eq:Phi}.  The estimates are uniform in $\delta$, so Remark~\ref{rem:singular-axes} removes the auxiliary truncation.

The next lemma is the resonance-fiber transversality statement.  It is the four-wave analogue of the preceding difference-map estimate.

\begin{lemma}\label{lem:transversality}
After shrinking $U$ if necessary, there is $c_1>0$ such that for all admissible $\xi,u,v$,
\begin{equation}\label{eq:gradPhi}
 |\nabla_u\Phi_{\xi,v}(u)|\ge c_1|v|.
\end{equation}
More precisely, for every $v\ne0$ one can choose a unit vector $e_v$, depending only on $v$ and the center Hessian, such that
\begin{equation}\label{eq:directionalPhi}
 |e_v\cdot\nabla_u\Phi_{\xi,v}(u)|\ge c_1|v|.
\end{equation}
The sign of this directional derivative is fixed on the patch.
\end{lemma}

\begin{proof}
With $H_0$ and $s_0$ as above, the fundamental theorem of calculus gives
\begin{equation}\label{eq:grad-average}
 \nabla_u\Phi_{\xi,v}(u)
 =-\left(\int_0^1D^2\phi(\xi+u+tv)\,dt\right)v.
\end{equation}
The averaged Hessian stays within $s_0/4$ of $H_0$, proving \eqref{eq:gradPhi}.  Taking
\[
 e_v=\frac{H_0v}{|H_0v|}
\]
and shrinking once more makes $e_v\cdot\nabla_u\Phi$ bounded away from zero with one fixed sign, proving \eqref{eq:directionalPhi}.
\end{proof}

The scale-critical geometric input is the short-chord fiber bound: the short chord controls the length of the resonance fiber, while the long chord controls the coarea denominator.

\begin{lemma}\label{lem:fiber-gain}
There is $C$ such that
\begin{equation}\label{eq:fiber-gain}
 \sup_{\xi,\,v\ne0}
 \int_{\substack{|u|\le |v|\\\Phi_{\xi,v}(u)=0}}
 \frac{d\cH^1(u)}{|\nabla_u\Phi_{\xi,v}(u)|}
 \le C.
\end{equation}
The symmetric estimate with $u$ and $v$ interchanged holds on $|v|<|u|$.
\end{lemma}

\begin{proof}
By Lemma~\ref{lem:transversality}, in coordinates whose first axis is $e_v$, the zero set of $\Phi_{\xi,v}$ is a uniformly Lipschitz graph over the orthogonal coordinate.  Its projection from the region $|u|\le |v|$ has length at most $2|v|$, while the directional derivative, and hence the full gradient, is bounded below by $c_1|v|$.  Coarea in this graph coordinate therefore gives a uniform bound.  The symmetric statement follows from
\[
 \Phi_{\xi,v}(u)=\Phi_{\xi,u}(v).
\]
\end{proof}

\begin{remark}[The quadratic model]\label{rem:quadratic-model}
If $\phi(\xi)=\frac12\xi^TH\xi$ with $H$ symmetric and invertible, then the resonance identity is exact:
\[
 \Phi_{\xi,v}(u)=-u^THv.
\]
Thus, for fixed $v$, the resonant $u$-fiber is a line perpendicular to $Hv$; its intersection with $|u|\le |v|$ has length $O(|v|)$ and the coarea denominator is $|Hv|\simeq_H |v|$.  Lemma~\ref{lem:fiber-gain} is therefore immediate in the quadratic case.  The preceding proof shows that on a sufficiently small curved patch the general $C^2$ geometry is a uniform perturbation of exactly this model.
\end{remark}

\begin{remark}[Degenerate axes]\label{rem:singular-axes}
The coarea representation of the expanded $L^2$ norm degenerates at $v=0$, and symmetrically at $u=0$.  These axes have zero measure in the product of the regular fibers.  Equivalently, one may first truncate $|u|,|v|\ge\delta$, use the two-region estimates below uniformly in $\delta$, and then pass to the limit.  The direct short-chord fiber bound is uniform down to zero; in particular the apparent $1/|v|$ singularity is locally integrable in the two-dimensional $v$ variable.
\end{remark}

\section{The two-chord proof of the local \texorpdfstring{$S^4$}{S4} extension theorem}\label{sec:localproof}

The next result is the local graph theorem.  It is the core analytic statement of the paper: once it is proved on a sufficiently small graph patch, the compact-surface theorem follows by a finite atlas and the ideal property of $\gamma$.

\begin{theorem}\label{thm:local}
Under the assumptions of Section~\ref{sec:geometry},
\begin{equation}\label{eq:local-main}
 \|E_\phi G\|_{L^4(\R^3;S^4)}
 \lesssim_\phi
 \max\left\{
 \left\|(\int_U G^*G)^{1/2}\right\|_4,
 \left\|(\int_U GG^*)^{1/2}\right\|_4
 \right\}.
\end{equation}
The estimate is uniform over finite matrix dimensions.
\end{theorem}

\begin{proof}
It suffices to treat smooth finite-rank $G$ compactly supported in $U$.  By the discussion preceding \eqref{eq:Kuniform}, it is enough to prove \eqref{eq:Kuniform}.  Expanding $\|K_\varepsilon\|_{L^2(S^2)}^2$, applying the coarea formula off the degenerate axes, and then inserting absolute values gives the resonant quartic integral associated with \eqref{eq:Phi}.  Split this integral into the two regions
\[
 |u|\le |v|
 \qquad\text{and}\qquad
 |v|<|u|.
\]
The restriction $|u|\ge\varepsilon$ only removes part of these regions, so all estimates below are uniform in $\varepsilon$.

\smallskip
\noindent\emph{Case 1: $|u|\le |v|$.}
Using cyclicity of the trace,
\begin{equation}\label{eq:column-pairing}
 \Tr(G_1^*G_2G_3^*G_4)
 =\langle G_1G_4^*,G_2G_3^*\rangle_{S^2}.
\end{equation}
Hence
\[
 |\Tr(G_1^*G_2G_3^*G_4)|
 \le
 \|G_1G_4^*\|_2\,\|G_2G_3^*\|_2.
\]
In the parametrization \eqref{eq:uv-param}, define
\[
 F_v(\xi)=\|G(\xi)G(\xi+v)^*\|_2.
\]
The two factors are $F_v(\xi)$ and $F_v(\xi+u)$.  The resonant measure is invariant under $(\xi,u)\mapsto(\xi+u,-u)$ because
\[
 \Phi_{\xi+u,v}(-u)=-\Phi_{\xi,v}(u).
\]
Cauchy--Schwarz, the symmetry above, and Lemma~\ref{lem:fiber-gain} therefore give
\begin{align}
 T_{\le}
 &\lesssim
 \int_{\R^2}\int_U
 \|G(\xi)G(\xi+v)^*\|_2^2
 \left[
 \int_{\substack{|u|\le|v|\\\Phi_{\xi,v}(u)=0}}
 \frac{d\cH^1(u)}{|\nabla_u\Phi_{\xi,v}(u)|}
 \right]d\xi\,dv\notag\\
 &\lesssim
 \iint_{U\times U}\|G(\xi)G(\eta)^*\|_2^2\,d\xi\,d\eta\notag\\
 &=\Tr\left(\int_U G^*G\,d\xi\right)^2
 =\left\|\left(\int_U G^*G\right)^{1/2}\right\|_4^4.
 \label{eq:column-energy}
\end{align}

\smallskip
\noindent\emph{Case 2: $|v|<|u|$.}
Use instead
\begin{equation}\label{eq:row-pairing}
 |\Tr(G_1^*G_2G_3^*G_4)|
 \le
 \|G_1^*G_2\|_2\,\|G_3^*G_4\|_2.
\end{equation}
For fixed $u$ put
\[
 H_u(\xi)=\|G(\xi)^*G(\xi+u)\|_2.
\]
The second factor is $H_u(\xi+v)$.  The symmetry $(\xi,v)\mapsto(\xi+v,-v)$, the symmetric part of Lemma~\ref{lem:fiber-gain}, and Cauchy--Schwarz give
\begin{align}
 T_{>}
 &\lesssim
 \int_{\R^2}\int_U
 \|G(\xi)^*G(\xi+u)\|_2^2
 \left[
 \int_{\substack{|v|<|u|\\\Phi_{\xi,v}(u)=0}}
 \frac{d\cH^1(v)}{|\nabla_v\Phi_{\xi,v}(u)|}
 \right]d\xi\,du\notag\\
 &\lesssim
 \iint_{U\times U}\|G(\xi)^*G(\eta)\|_2^2\,d\xi\,d\eta\notag\\
 &=\Tr\left(\int_U GG^*\,d\xi\right)^2
 =\left\|\left(\int_U GG^*\right)^{1/2}\right\|_4^4.
 \label{eq:row-energy}
\end{align}
Combining \eqref{eq:column-energy} and \eqref{eq:row-energy} proves \eqref{eq:Kuniform} uniformly in $\varepsilon$.

Choose $\varepsilon_k\downarrow0$.  The uniform $L^2(\R^3;S^2)$ bound and Hilbert-space weak compactness give a weakly convergent subsequence $K_{\varepsilon_k}\rightharpoonup K$.  By \eqref{eq:Keps-fourier}, the same sequence converges distributionally to the Fourier transform of $(E_\phi G)^*E_\phi G$, so these distributions coincide.  By weak lower semicontinuity, the limit satisfies the same $L^2$ bound.  Plancherel now yields
\[
 \|(E_\phi G)^*E_\phi G\|_{L^2(S^2)}^2
 \lesssim
 \left\|\left(\int_U G^*G\right)^{1/2}\right\|_4^4
 +
 \left\|\left(\int_U GG^*\right)^{1/2}\right\|_4^4.
\]
Finally,
\[
 \|E_\phi G\|_{L^4(S^4)}^4
 =\|(E_\phi G)^*E_\phi G\|_{L^2(S^2)}^2,
\]
which proves \eqref{eq:local-main}.  Density gives the stated completion.
\end{proof}

\begin{remark}[Why the crossing obstruction disappears]\label{rem:crossing}
If one globally fixes the first pairing in \eqref{eq:row-pairing}, one is naturally led to a reordered quartic term plus a noncommutative commutator.  Theorem~\ref{thm:local} does not show that this commutator vanishes.  It shows that estimating it separately is unnecessary.  The operator ordering is chosen locally according to whether $|u|\le|v|$ or $|v|<|u|$.  This is the noncommutative analogue of allowing the geometry to choose the favorable bilinear factorization; no dyadic decomposition is needed once the two regions are separated.
\end{remark}

\section{Global curved surfaces}\label{sec:global}

\begin{proof}[Proof of Theorem~\ref{thm:main-extension}]
Cover $\Sigma$ by finitely many coordinate patches $\Sigma_j$ on which, after a rigid rotation, the surface is the graph of a $C^2$ function $\phi_j$ whose Hessian is uniformly nonsingular.  The nonvanishing Gaussian-curvature assumption guarantees this after shrinking the patches.  Choose a smooth partition of unity $\sum_j\chi_j=1$ subordinate to this cover, with $\supp\chi_j\Subset\Sigma_j$.

In graph coordinates the surface measure has the form
\[
 d\sigma=w_j(\xi)\,d\xi
\]
with $w_j$ positive and bounded above and below.  Let
\[
 U_j:L^2(\Sigma_j,d\sigma)\to L^2(U_j,d\xi),
 \qquad (U_jG)(\xi)=w_j(\xi)^{1/2}G(\Gamma_j(\xi)),
\]
which is unitary.  The localized surface extension is then exactly the graph extension applied to
\[
 w_j^{1/2}U_j(\chi_jG).
\]
Multiplication by the bounded scalar function $w_j^{1/2}$ is a bounded operator on the Hilbert variable $L^2(U_j)$.  Hence the ideal property of $\gamma$ gives
\[
 \|w_j^{1/2}U_j(\chi_jG)\|_{\ga(L^2(U_j),S^4)}
 \le \|w_j^{1/2}\|_\infty
 \|\chi_jG\|_{\ga(L^2(\Sigma),S^4)}
 \lesssim_j\|G\|_\ga.
\]
The local theorem and the triangle inequality in $L^4(S^4)$ therefore give
\[
 \|E_\Sigma G\|_{L^4(S^4)}
 \le\sum_j\|E_\Sigma(\chi_jG)\|_{L^4(S^4)}
 \lesssim_\Sigma\|G\|_\ga.
\]
Finite-rank operators are dense in $\ga$ by definition; approximating the finitely many scalar $L^2$ coefficient functions by smooth functions supported inside the coordinate patches shows that the smooth finite-rank surface amplitudes used above form a dense core.  This completes the proof.
\end{proof}

\begin{remark}[Semifinite noncommutative $L^4$ spaces]\label{rem:semifinite}
The local proof only uses the tracial Hilbert space $L^2(\cM,\tau)$, H\"older multiplication $L^4\cdot L^4\subset L^2$, and the row/column form of noncommutative Khintchine.  The same argument therefore extends, by standard approximation, from Schatten classes to semifinite noncommutative $L^4(\cM,\tau)$ spaces.  We state the Schatten version as the main theorem because it is the sharp classification needed below and avoids additional measure-topology notation.
\end{remark}

\section{Restriction duality and the sharp Schatten range}\label{sec:classification}

\subsection{The critical dual theorem}
For $1<a<\infty$, Schatten classes are $K$-convex.  The trace pairing identifies the dual of $\ga(H,S^a)$ with $\ga(H,S^{a'})$ up to constants depending only on $a$; see \cite{vanNeerven2010}.  Applying this to Theorem~\ref{thm:main-extension} proves Corollary~\ref{cor:critical-restriction}.

\subsection{The large-Schatten branch}
Scalar Stein--Tomas gives
\[
 E_\Sigma:L^2(\Sigma)\to L^4(\R^3).
\]
Proposition~\ref{prop:factorization} therefore gives
\[
 E_\Sigma:\ga(L^2(\Sigma),S^b)\to L^4(\R^3;S^b),
 \qquad 1<b\le2.
\]
Dualizing at $b=a'$ yields
\begin{equation}\label{eq:largeS}
 R_\Sigma:L^{4/3}(\R^3;S^a)\to\ga(L^2(\Sigma),S^a),
 \qquad 2\le a<\infty.
\end{equation}
Interpolation with Proposition~\ref{prop:L1} yields the full $1\le p\le4/3$ range for $a\ge2$.

\subsection{Interpolation below \texorpdfstring{$S^2$}{S2}}
We briefly spell out the interpolation because the domain and coefficient exponents move together.

For $4/3\le a\le2$, interpolate at the fixed spatial exponent $4/3$ between Corollary~\ref{cor:critical-restriction} and the Hilbert-valued endpoint
\[
 R_\Sigma:L^{4/3}(\R^3;S^2)\to\ga(L^2(\Sigma),S^2).
\]
For a finite orthonormal family $(h_j)_{j=1}^N\subset L^2(\Sigma)$, compose $R_\Sigma$ with the Gaussian coefficient map
\[
 T_{N,h}f
 =\sum_{j=1}^N\gamma_j
 \int_\Sigma \widehat f(\omega)\overline{h_j(\omega)}\,d\sigma(\omega).
\]
The two endpoint estimates are bounds into $L^2(\Omega;S^{4/3})$ and $L^2(\Omega;S^2)$, uniformly in $N$ and in the orthonormal family.  Ordinary complex interpolation applies because
\[
 [L^{4/3}(\R^3;S^{4/3}),L^{4/3}(\R^3;S^2)]_\theta
 =L^{4/3}(\R^3;S^a)
\]
and
\[
 [L^2(\Omega;S^{4/3}),L^2(\Omega;S^2)]_\theta
 =L^2(\Omega;S^a),
\]
where $1/a=(1-\theta)/(4/3)+\theta/2$.  Thus the finite Gaussian operators have the $S^a$ bound with the same uniformity.  Taking the supremum over finite orthonormal systems yields the $\gamma_\infty$ norm.  Since $S^a$ is reflexive for $1<a<\infty$ and hence contains no copy of $c_0$, the Hoffmann--J\o rgensen--Kwapie\'n theorem identifies $\gamma_\infty$ with $\gamma$; see \cite{vanNeerven2010}.  Thus
\begin{equation}\label{eq:midS}
 R_\Sigma:L^{4/3}(S^a)\to\ga(S^a),
 \qquad 4/3\le a\le2.
\end{equation}
Interpolation with $p=1$ gives all smaller $p$.

For $1<a<4/3$, apply the same finite-Gaussian argument between the universal endpoint
\[
 L^1(\R^3;S^1)\to\ga(S^1)
\]
and the critical $(p,a)=(4/3,4/3)$ theorem.  Interpolation of the domain Bochner spaces and the finite Gaussian targets gives the line $p=a$: indeed the same parameter satisfies $1/p=1-\theta/4=1/a$.  After taking the supremum, $\gamma_\infty=\gamma$ because $S^a$ is reflexive.  Hence
\begin{equation}\label{eq:smallSdiag}
 R_\Sigma:L^a(\R^3;S^a)\to\ga(S^a),
 \qquad 1<a<4/3.
\end{equation}
Interpolating once more with the universal $p=1$ estimate gives $1\le p\le a$.

\subsection{Necessity}
The scalar rank-one test $f(x)=g(x)x_0$ shows that \eqref{eq:classification} cannot hold beyond the scalar Stein--Tomas exponent $p=4/3$.  To see the coefficient-space obstruction, one may use the standard physical-translate argument.  Choose a compactly supported scalar $g$ whose Fourier transform does not vanish on a surface patch, and translations $y_j$ so far apart that the functions
\[
 e^{-iy_j\cdot\omega}\widehat g(\omega)
\]
form a uniformly Riesz sequence in $L^2(\Sigma)$.  For
\[
 f(x)=\sum_jg(x-y_j)x_j,
\]
disjoint physical supports give
\[
 \|f\|_{L^p(S^a)}
 =\|g\|_p\left(\sum_j\|x_j\|_a^p\right)^{1/p},
\]
while a bounded left inverse in the Hilbert variable of the $\gamma$ norm yields Gaussian type $p$ for $S^a$.  Since $S^a$ has maximal type $a$ for $a<2$, one must have $p\le a$.  This proves the sharpness in Theorem~\ref{thm:classification}.

\subsection{Type alone does not characterize the general Banach endpoint}\label{subsec:lorentz-counterexample}
The Schatten theorem reaches the coefficient-space type threshold, but this is a special feature of the Schatten geometry rather than a consequence of type alone.  The following proposition gives a concrete contrast.

\begin{proposition}\label{prop:lorentz-counterexample}
Let $2<q<\infty$ and let
\[
 X=L^{4/3,q}((0,1)^3)
\]
be equipped with an equivalent Banach lattice norm.  Then $X$ has Gaussian type $4/3$, but the spherical endpoint
\[
 R_{S^2}:L^{4/3}(\R^3;X)\longrightarrow\ga(L^2(S^2),X)
\]
is not bounded.
\end{proposition}

\begin{proof}
The type assertion follows from the classical Lorentz-space theorem of Creekmore \cite{Creekmore1981}: since $4/3<q$ and $4/3<2$, the space $L^{4/3,q}$ has Rademacher type $4/3$, equivalently Gaussian type $4/3$.

Suppose, to the contrary, that the displayed restriction estimate holds.  Let $f\in C_c^\infty(\R^3)$ and put $\Omega=(0,1)^3$.  For $L\ge1$ define
\[
 F_L(x)(y)=f(x-Ly),\qquad y\in\Omega.
\]
For every fixed $x$, the distribution function of $y\mapsto f(x-Ly)$ is bounded by $L^{-3}$ times that of $f$.  Hence the Lorentz dilation law and monotonicity give
\[
 \|F_L(x)\|_X\lesssim_q L^{-9/4}\|f\|_{L^{4/3,q}(\R^3)}.
\]
If $\supp f\subset[-R,R]^3$, then $F_L$ is supported in $L\Omega+[-R,R]^3$ as a function of $x$, so
\begin{equation}\label{eq:lorentz-translate-domain}
 \limsup_{L\to\infty}
 \|F_L\|_{L^{4/3}(\R^3;X)}
 \lesssim_q
 \|f\|_{L^{4/3,q}(\R^3)}.
\end{equation}
On the sphere,
\[
 (R_{S^2}F_L)(\omega)(y)
 =e^{-iLy\cdot\omega}\widehat f(\omega).
\]
For a Banach function lattice the Gaussian norm dominates the pointwise square-function norm.  Indeed, for a finite orthonormal family $(h_j)$ and coefficient functions $x_j\in X$,
\begin{align*}
 \left(\E\left\|\sum_j\gamma_jx_j\right\|_X^2\right)^{1/2}
 &\ge \E\left\|\sum_j\gamma_jx_j\right\|_X\\
 &\ge \left\|\E\left|\sum_j\gamma_jx_j\right|\right\|_X
 =c\left\|\left(\sum_j|x_j|^2\right)^{1/2}\right\|_X.
\end{align*}
Applying this to finite orthogonal projections in $L^2(S^2)$, using the $\gamma$-ideal property and then monotone convergence, gives
\[
 \|R_{S^2}F_L\|_{\ga(L^2(S^2),X)}
 \gtrsim_q
 \left\|\left(\int_{S^2}|(R_{S^2}F_L)(\omega)(\cdot)|^2\,d\sigma(\omega)\right)^{1/2}\right\|_X.
\]
The square function on the right is the constant function $\|R_{S^2}f\|_2$.  Combining this observation with \eqref{eq:lorentz-translate-domain} and letting $L\to\infty$ would therefore imply the scalar Lorentz estimate
\[
 \|R_{S^2}f\|_2\lesssim_q\|f\|_{L^{4/3,q}(\R^3)}.
\]
Bak and Seeger proved the stronger positive endpoint $L^{4/3,2}\to L^2(S^2)$ and showed that the second Lorentz index $2$ is optimal: the estimate fails for every $L^{4/3,q}$ with $q>2$ \cite{BakSeeger2011}.  Their finite multiscale Knapp examples may be truncated smoothly in physical space without changing the contradiction, because they are rapidly decreasing and $L^1\to L^2(S^2)$ restriction is bounded.  This contradiction proves the proposition.
\end{proof}

\begin{remark}\label{rem:lorentz-significance}
Proposition~\ref{prop:lorentz-counterexample} shows that the general Banach-space endpoint contains information beyond an attained Gaussian type exponent.  In particular, the positive Schatten result cannot be recovered from a type-only transfer principle.  The conclusion is already visible on finite-dimensional coefficient spaces: for every fixed $q>2$, the translation construction may be approximated in the internal variable by step functions and in the physical variable in both $L^{4/3}$ and $L^1$.  Proposition~\ref{prop:L1} controls the resulting output error.  Hence one obtains finite-dimensional subspaces $X_N\subset L^{4/3,q}$ whose type-$4/3$ constants stay uniformly bounded while their best spherical endpoint restriction constants tend to infinity.  Thus no dimension-free endpoint theorem can depend only on the type constant.

The same argument singles out $L^{4/3,2}$ as a natural Lorentz borderline.  Bak--Seeger give the scalar estimate $L^{4/3,2}(\R^3)\to L^2(S^2)$, so the translation obstruction above disappears at the second Lorentz index $2$; that scalar fact, however, does not by itself imply the $L^{4/3,2}$-valued $\gamma$ estimate.
\end{remark}

\section{Classical restriction theory as a special and neighboring theory}\label{sec:classical}

The purpose of this section is twofold: to make explicit which classical statements are direct consequences of our theorem, and to distinguish them from neighboring results that motivate but do not formally follow from it.

\subsection{Direct scalar recovery}
Take $S^a=\C$, or equivalently restrict Theorem~\ref{thm:main-extension} to a fixed rank-one operator direction.  Then $C_G=R_G=\|G\|_{L^2}$ and Theorem~\ref{thm:main-extension} becomes
\[
 \|E_\Sigma g\|_{L^4(\R^3)}\lesssim\|g\|_2.
\]
Duality gives the three-dimensional Stein--Tomas endpoint
\[
 \|\widehat f|_\Sigma\|_2\lesssim\|f\|_{4/3},
\]
and interpolation with $L^1$ gives $1\le p\le4/3$.  Thus the classical Tomas--Stein theorem on compact curved surfaces in $\R^3$ is literally the scalar subcase of the present result.

For $S^2$ with the standard normalization, sharp constants and extremizers constitute a separate sharp-constant problem; see Foschi and Oliveira e Silva \cite{Foschi2015,Oliveira2024}.  Our theorem is dimension-free in the internal matrix size, not a claim about the optimal scalar geometric constant.

\subsection{Fefferman, Tomas, Stein, and Carleson--Sj\"olin}
Stein first observed that curvature can permit $L^p\to L^2$ restriction to the sphere; early Fefferman--Stein bounds are recorded in Fefferman's 1970 paper \cite{Fefferman1970}.  Tomas \cite{Tomas1975} obtained the full nonendpoint range in all dimensions and ended his note by crediting Stein with the endpoint extension by complex interpolation; see \cite{Oliveira2024} for a concise historical account.  Fefferman's ball-multiplier theorem \cite{Fefferman1971} then displayed the complementary force of Kakeya geometry in Fourier multiplier theory.  Stein subsequently developed the oscillatory-integral framework in which restriction is now usually presented \cite{Stein1993}.  Carleson--Sj\"olin \cite{CarlesonSjolin1972} replaced translation-invariant extension by variable-phase oscillatory integrals.  Bak and Seeger later proved endpoint Lorentz refinements of Stein--Tomas and parallel endpoint improvements for Carleson--Sj\"olin--H\"ormander operators and spectral projections \cite{BakSeeger2011}.  Theorem~\ref{thm:main-extension} should therefore be viewed as a coefficient-space generalization of the three-dimensional Stein--Tomas endpoint, not as a generalization of the full variable-phase theory.

\subsection{Strichartz and dispersive equations}
Strichartz's 1977 paper \cite{Strichartz1977} treated broad classes of quadratic surfaces and explicitly connected restriction/extension with spacetime decay for solutions of wave and Klein--Gordon equations.  The paraboloid yields Schr\"odinger evolution.  Keel and Tao \cite{KeelTao1998} later established an abstract dispersive theorem giving the endpoint admissible family for a broad class of propagators.

Our Corollary~\ref{cor:schrodinger} is a coefficient-space analogue of one particularly symmetric classical Strichartz estimate: the $L^4_{t,x}$ estimate in two spatial dimensions.  It is not the forbidden two-dimensional endpoint $L_t^2L_x^\infty$, nor does it by itself give the full Keel--Tao mixed-norm family.  What is new is that the initial frequency amplitude may be genuinely noncommuting and is measured in the $\gamma(L^2,S^4)$ norm rather than the Bochner $L^2(S^4)$ norm.

\subsection{Bilinear, multilinear, and polynomial-partitioning restriction}
The decisive geometric idea in our proof is bilinear in spirit: a quartic expression is factored into two Hilbert--Schmidt products, and the favorable factorization is chosen according to the relative lengths of two frequency chords.  This is philosophically close to the bilinear restriction program of Tao--Vargas--Vega \cite{TaoVargasVega1998} and Tao \cite{Tao2003}, and to the multilinear transversality theory of Bennett--Carbery--Tao \cite{BCT2006}.  Bourgain--Guth \cite{BourgainGuth2011} and Guth \cite{Guth2018} use multilinear information and polynomial partitioning to go beyond the Stein--Tomas line in the scalar restriction conjecture.  Our theorem does not improve the scalar restriction exponents; instead it uses the special even exponent $4$ to obtain a new noncommutative amplification exactly at the Stein--Tomas point.

\subsection{Sogge, spectral clusters, and Kenig--Ruiz--Sogge}
Restriction estimates are local Euclidean models for spectral projection and Fourier integral estimates on manifolds.  Sogge's spectral-cluster bounds \cite{Sogge1988} are a fundamental manifestation of this principle.  Kenig--Ruiz--Sogge \cite{KRS1987} developed uniform Sobolev inequalities and unique continuation estimates in which the same curved-frequency geometry reappears through the resolvent.  Section~\ref{sec:PDE} derives a free spectral-measure consequence of our theorem.  A full $\gamma$/Schatten KRS resolvent theorem requires an additional singular-integral/cancellation argument and is intentionally not claimed here.

\subsection{Orthonormal systems and trace ideals}
Frank--Lewin--Lieb--Seiringer \cite{FrankLewinLiebSeiringer2014} and Frank--Sabin \cite{FrankSabin2017} developed restriction and Strichartz estimates for orthonormal systems, naturally expressed through Schatten trace ideals.  Those estimates are powerful in many-body and density-matrix PDE.  The present theorem concerns a different noncommutative object.  Rather than a scalar density
\[
 \sum_j\nu_j|E h_j|^2,
\]
we estimate a coherent matrix field
\[
 \sum_j(E h_j)x_j
\]
with arbitrary noncommuting coefficients $x_j$.  The right norm simultaneously sees
\[
 \sum_jx_j^*x_j
 \quad\text{and}\quad
 \sum_jx_jx_j^*.
\]
This coherent row/column geometry is invisible in scalar occupation-number estimates.

\section{Applications to PDE and spectral theory}\label{sec:PDE}

\subsection{Matrix Herglotz waves and Helmholtz}
For $G:S^2\to S^4$, define the matrix Herglotz wave
\begin{equation}\label{eq:Herglotz}
 U(x)=\int_{S^2}e^{ix\cdot\omega}G(\omega)\,d\sigma(\omega).
\end{equation}
Entrywise,
\[
 (\Delta+1)U=0.
\]
Theorem~\ref{thm:main-extension} gives
\begin{equation}\label{eq:Herglotz-bound}
 \boxed{
 \|U\|_{L^4(\R^3;S^4)}
 \lesssim
 \|G\|_{\ga(L^2(S^2),S^4)}.
 }
\end{equation}
This controls coherent multichannel Helmholtz fields.  In scattering language, the angular data may represent coupled polarization or channel amplitudes; noncommutativity allows interactions that cannot be diagonalized once and for all in the angular variable.

By duality, a compactly supported $S^{4/3}$-valued source has a $\gamma$-radonifiable far-field restriction:
\[
 \|\widehat f|_{S^2}\|_{\ga(L^2(S^2),S^{4/3})}
 \lesssim\|f\|_{L^{4/3}(\R^3;S^{4/3})}.
\]
The scalar far-field estimate is the rank-one special case.

\subsection{Schr\"odinger evolution in two space dimensions}
For the paraboloid
\[
 \phi(\xi)=|\xi|^2,
\]
the resonance function satisfies exactly
\[
 \Phi_{\xi,v}(u)=-2u\cdot v.
\]
Thus the fiber is a line perpendicular to $v$ and
\[
 \delta(u\cdot v)\,du
 =|v|^{-1}d\cH^1_{v^\perp}(u).
\]
Theorem~\ref{thm:local} therefore holds globally for compactly supported data without a curvature-perturbation error.  After the harmless Fourier normalization this is Corollary~\ref{cor:schrodinger}.

The significance of \eqref{eq:schrodinger} is not merely ``matrix-valued Strichartz.''  A generic $S^4$-valued $L^2_x$ function and a $\gamma(L^2_x,S^4)$ element have different geometry.  The $\gamma$ norm records randomized coherent superpositions and is the natural input norm for cylindrical Gaussian forcing and for many Banach-space stochastic integration problems.

\subsection{Saddle dispersion and negative curvature}
The hyperbolic model \eqref{eq:hyperbolic-model} gives a second exact quadratic test.  If
\[
 V(t,x)=\int_{\R^2}e^{i(x\cdot\xi+t(\xi_1^2-\xi_2^2))}G(\xi)\,d\xi,
\]
then the resonance function is
\[
 \Phi_{\xi,v}(u)=-2(u_1v_1-u_2v_2),
\]
so the coarea fiber is again a line and the two-chord proof applies without perturbative geometry.  For compactly supported $S^4$-valued data,
\begin{equation}\label{eq:hyperbolic-L4}
 \boxed{
 \|V\|_{L^4_{t,x}(S^4)}
 \lesssim
 \|G\|_{\gamma(L^2(\R^2),S^4)}.
 }
\end{equation}
Equivalently, this is an $L^4$ estimate for the matrix-valued saddle-dispersive equation $i\partial_tV+(\partial_{x_1}^2-\partial_{x_2}^2)V=0$.  Scalar hyperbolic restriction is part of the classical theory, including modern results for general smooth negative-curvature surfaces \cite{BuschenhenkeMullerVargas2023}; \eqref{eq:hyperbolic-L4} displays the additional coherent row/column information retained by the Schatten-valued theorem.

\subsection{Frequency-localized Klein--Gordon and other nondegenerate graphs}
Let
\[
 \phi_m(\xi)=\sqrt{m^2+|\xi|^2}.
\]
On every compact frequency annulus, $D^2\phi_m$ is nonsingular, so Theorem~\ref{thm:local} gives a frequency-localized $S^4$-valued $L^4_{t,x}$ extension estimate for the two-dimensional Klein--Gordon flow.  More generally, any compactly supported dispersive graph with uniformly nonsingular Hessian admits the same local theorem.

The cone is different: one principal curvature vanishes.  Therefore the wave equation requires cone-specific restriction geometry and is not covered by simply setting $\phi(\xi)=|\xi|$ in Theorem~\ref{thm:local}.  This distinction mirrors the classical literature beginning with Strichartz \cite{Strichartz1977}.

\subsection{A free spectral-measure consequence}
Let $P=\sqrt{-\Delta}$ on $\R^3$.  With the unit-sphere normalization, write
\[
 R_\lambda f(\omega)=\widehat f(\lambda\omega),
 \qquad
 E_\lambda g(x)=\int_{S^2}e^{i\lambda x\cdot\omega}g(\omega)\,d\sigma(\omega).
\]
Scaling gives
\[
 \|R_\lambda\|_{L^{4/3}(S^{4/3})\to\ga(S^{4/3})}
 \lesssim \lambda^{-3/4},
\]
and Proposition~\ref{prop:factorization} gives the matching extension bound with $S^{4/3}$ values.  Since
\[
 \frac{dE_P}{d\lambda}(\lambda)
 =c\lambda^2E_\lambda R_\lambda,
\]
we obtain
\begin{equation}\label{eq:spectral-P}
 \boxed{
 \left\|\frac{dE_P}{d\lambda}(\lambda)\right\|_{L^{4/3}(S^{4/3})\to L^4(S^{4/3})}
 \lesssim \lambda^{1/2}.
 }
\end{equation}
For $A=-\Delta$ and energy $\rho=\lambda^2$,
\begin{equation}\label{eq:spectral-A}
 \boxed{
 \left\|\frac{dE_A}{d\rho}(\rho)\right\|_{L^{4/3}(S^{4/3})\to L^4(S^{4/3})}
 \lesssim \rho^{-1/4}.
 }
\end{equation}
The exponent is the same scaling exponent that appears on the conjugate line in free uniform Sobolev theory.  However, integrating the spectral measure against $(\rho-z)^{-1}$ requires principal-value cancellation near the spectrum.  Equations \eqref{eq:spectral-P}--\eqref{eq:spectral-A} are therefore a precursor to, not a substitute for, a KRS-type resolvent theorem.

\subsection{Radonification of Gaussian surface fields}
Let $K$ be a Hilbert noise space and
\[
 \Psi\in\ga\big(K,\ga(L^2(\Sigma),S^4)\big).
\]
By the ideal property \eqref{eq:gamma-ideal} and Theorem~\ref{thm:main-extension},
\begin{equation}\label{eq:stochastic}
 E_\Sigma\circ\Psi
 \in\ga\big(K,L^4(\R^3;S^4)\big),
 \qquad
 \|E_\Sigma\Psi\|_\ga
 \lesssim\|\Psi\|_\ga.
\end{equation}
Thus a cylindrical Gaussian angular field whose covariance operator is radonifying at the surface level becomes an actual $L^4_x(S^4)$-valued Gaussian Herglotz field.  This is the safe stochastic formulation: homogeneous spectral white noise with constant pointwise variance is not globally $L^4_x$ without an additional radonifying/spatial localization assumption.

\subsection{Why the noncommutative extension can matter}
The classical scalar estimate treats one amplitude at a time.  Hilbert tensorization treats many channels when their internal geometry is orthogonal.  Orthonormal-density inequalities treat incoherent occupation numbers.  Theorem~\ref{thm:main-extension} instead controls coherent sums with noncommuting internal amplitudes.  This is relevant whenever the field itself is matrix/operator valued, for example:
\begin{itemize}
\item finite or countable channel scattering with noncommuting coupling matrices;
\item density-operator or covariance amplitudes before taking a scalar trace;
\item matrix-valued random waves with correlated channels;
\item linearized systems in which spectral polarization changes with frequency and cannot be simultaneously diagonalized.
\end{itemize}
The theorem does not solve these PDEs by itself.  It provides a dimension-independent linear estimate in exactly the coefficient geometry needed to begin perturbative, stochastic, or nonlinear arguments without first discarding coherent operator structure.

\section{Further directions}

\subsection{Banach-space geometry}
For a general Banach space $X$, $\gamma$-valued restriction at exponent $p$ forces Gaussian type $p$, and on bounded flat frequency sets type $p$ is sufficient.  Proposition~\ref{prop:lorentz-counterexample} shows that curvature does not make type alone sufficient at the three-dimensional Stein--Tomas endpoint: $L^{4/3,q}$ has type $4/3$ for $q>2$ but fails the endpoint restriction estimate.  Thus a general characterization must contain finer information reflecting how curvature-generated contributions from different scales are aggregated.  The Schatten theorem shows that noncommutative row/column geometry supplies enough additional structure to reach the type threshold in that scale.  The borderline space $L^{4/3,2}$ is an especially concrete next test.  We do not propose a general characterization here; identifying an intrinsic Banach-space property behind endpoint $\gamma$-restriction is a separate problem.

\subsection{Variable phases}
The next natural extension is a $\gamma(S^4)\to L^4(S^4)$ theorem for Carleson--Sj\"olin/H\"ormander oscillatory integrals with quantitatively controlled phase geometry.  The local proof here suggests the correct mechanism: resonance fibers, two adjacent increments, and a scale-dependent choice of row or column factorization.  Carrying this out uniformly under variable phase localization is a separate theorem.

\subsection{Uniform resolvents}
The spectral-measure estimate \eqref{eq:spectral-A} has the KRS scaling, but a resolvent requires the cancellation of a singular Cauchy integral.  A future Schatten/$\gamma$ KRS theorem must therefore add a limiting-absorption or analytic-family argument.  Frank--Sabin's trace-ideal uniform Sobolev theory \cite{FrankSabin2017} provides an important neighboring model, while the present theorem supplies a different coherent coefficient geometry.

\subsection{Higher dimensions and even exponents}
In dimension two the Stein--Tomas dual exponent for the circle is $6$.  A three-fold alternating product of $S^6$ coefficients again lands in $S^2$, suggesting a trilinear analogue of the present proof.  More generally, even exponents may permit a hierarchy of noncommutative multilinearizations.  The combinatorics of resonance fibers becomes substantially more complicated, and we do not pursue it here.

\section*{Acknowledgments and disclosure}
The authors received no external funding for this work.

The authors used Gemini for language, organization and bibliographic assistance.  The authors take full responsibility for the mathematical content and the cited sources.

\end{document}